\documentclass{aptpub}

\authornames{Deng, Schilling and Song} 
\shorttitle{Sub-geometric rates of convergence under subordination} 

\usepackage{amsmath, verbatim, color}
\usepackage{mathrsfs}                                   
\usepackage{dsfont}

\newcommand\EE{\mathds E}
\newcommand\NN{\mathds N}
\newcommand\R{\mathds R}
\newcommand\Z{\mathds Z}

\newcommand\I{\mathds 1}
\newcommand\Prob{\mathds P}

\newcommand\dup{\mathrm{d}}
\newcommand\eup{\mathrm{e}}                                

\numberwithin{equation}{section}

\begin{document}

\title{Subgeometric rates of convergence\\
for Markov processes under subordination} 

\authorone[Wuhan University]{Chang-Song Deng} 
\addressone{School of Mathematics
and Statistics, Wuhan University, Wuhan 430072, China.
Email address: dengcs@whu.edu.cn} 

\authortwo[TU Dresden]{Ren\'e L.\ Schilling}
\addresstwo{TU Dresden, Fachrichtung Mathematik,
Institut f\"ur Mathematische Stochastik, 01062 Dresden, Germany.
Email address: rene.schilling@tu-dresden.de}

\authorthree[Zhongnan University of Economics and Law]{Yan-Hong Song} 
\addressthree{School of Statistics and Mathematics, Zhongnan University of Economics and Law, Wuhan 430073, China.
Email address: songyh@znufe.edu.cn}

\footnote{\textbf{Notice:} The paper is published in \emph{Adv.\ Appl.\ Prob.} \textbf{49} (2017) 162--181. The present arXiv version v3 contains the small corrections (Lemma~\ref{lohgr}, proof of Theorem~\ref{moments} c-ii) on page~\pageref{bigtime}) mentioned in an erratum.}

\begin{abstract}
    We are interested in the rate of convergence of a subordinate Markov process to its invariant measure. Given a subordinator and the corresponding Bernstein function (Laplace exponent) we characterize the convergence rate of the subordinate Markov process; the key ingredients are the rate of convergence of the original process and the (inverse of the) Bernstein function. At a technical level, the crucial point is to bound three types of
    moments (sub-exponential, algebraic and logarithmic) for
    subordinators as time $t$ tends to infinity. At the end we discuss some concrete models and we show that subordination can dramatically change the speed of convergence to equilibrium.
\end{abstract}



\section{Introduction and main result}\label{sec1} 



The notion of stochastic stability of a Markov process is of fundamental importance both in theoretical studies and in practical applications. There are various characterizations of geometric (or exponential) ergodicity, see e.g.\ \cite{MT93, MT93a, MT93b, che04, Chen05, DMT95} and the references given in these papers. In recent years, there has been considerable interest in sub-geometric (or sub-exponential) ergodicity, see for example \cite{DFMS04, for05, DMS07, dou09, But14, RW01} for developments in this direction. In this paper, we want to study sub-geometric convergence rates of Markov processes under subordination in the sense of Bochner.

Let $X=\{X_t:t\geq0\}$ be a Markov process with state space $(E,\mathscr{B}(E))$ and transition function $P^t(x,\dup y)$.
We assume that $E$ is a locally compact and separable metric space and we denote by $\mathscr{B}(E)$ the corresponding Borel $\sigma$-algebra. Let $f:E\to[1,\infty)$ be a measurable control function; the \emph{$f$-norm} of a signed measure $\mu$ on $E$ is defined as $\|\mu\|_f:=\sup_{|g|\leq f}|\mu(g)|$. Here, the supremum ranges over all measurable $g$ which are dominated by $f$ and $\mu(g) := \int g\,\dup\mu$. It is not hard to see that $\|\cdot\|_f\geq\|\cdot\|_{\operatorname{TV}}$ always holds for the total variation norm $\|\cdot\|_{\operatorname{TV}}$; if $f$ is bounded, then the norms $\|\cdot\|_f$ and $\|\cdot\|_{\operatorname{TV}}$ are even equivalent.

The convergence behaviour of a process $X$ to a stationary distribution $\pi$ in the $f$-norm can be captured by estimates of the form
\begin{equation}\label{rate1}
    \left\|P^t(x,\cdot)-\pi\right\|_f
    \leq C(x)r(t), \quad x\in E,\; t\geq 0,
\end{equation}
where $C(x)\in (0,\infty)$ is a constant depending on $x\in E$ and  $r:[0,\infty)\to(0,1]$ is the non-increasing \emph{rate function}. We say that $X$ displays \emph{sub-geometric convergence in $f$-norm}, if the rate function $r$ satisfies $r(t)\downarrow 0$ and $\log r(t)/t \uparrow 0$ as $t\to\infty$. Such $r$ are called \emph{sub-geometric rates}.

In many cases, the convergence rate $r$ can be explicitly
given, and the typical examples are
\begin{equation}\label{typical}
    r(t)=\eup^{-\theta t^\delta},\quad
    r(t)=(1+t)^{-\beta},\quad
    r(t)=\left[1+\log(1+t)\right]^{-\gamma},
\end{equation}
where $\theta>0$, $\delta\in(0,1]$ and $\beta,\gamma>0$ are some constants; see Section~\ref{sec4} below for specific models. Note that $r(t)=\eup^{-\theta t}$ is the classical exponential convergence rate. Some authors refer to the above examples as  sub-exponential, algebraic and logarithmic rates, respectively.

Bochner's subordination is a means to obtain more general (and also interesting) jump--type Markov processes from a given Markov process through a random time change by an independent non-decreasing L\'evy process (a subordinator). Among the most interesting examples are the symmetric $\alpha$-stable L\'{e}vy processes, which can be
viewed as subordinate Brownian motions. It is known that many fine properties of Markov processes (and the corresponding Markov semigroups) are preserved under subordination, see \cite{GRW11, DS15b} for Harnack and shift Harnack inequalities for subordinate semigroups, \cite{SW12, GM15} for Nash and Poincar\'{e} inequalities under subordination, and \cite{DS15a} for the quasi-invariance property of subordinate Brownian motion.

Let us recall the basics of Bochner's subordination. Let $S=\{S_t:t\geq0\}$ be a \emph{subordinator} (without killing), i.e.\ a non-decreasing L\'{e}vy process on $[0,\infty)$ with Laplace transform
$$
    \EE\,\eup^{-uS_t}=\eup^{-t\phi(u)},\quad u>0,\;t\geq0.
$$
The characteristic (Laplace) exponent $\phi:(0,\infty)\to (0,\infty)$ is a \emph{Bernstein function}, i.e.\ $\phi$ is of class $C^\infty$ such that $(-1)^{n-1}\phi^{(n)}\geq 0$ for all $n=1,2,\dots$; it is well known that every Bernstein function admits a unique (L\'evy--Khintchine) representation
\begin{equation}\label{bern-rep}
    \phi(u)=bu+\int_{(0,\infty)}\left(1-\eup^{-uy}\right)
    \,\nu(\dup y),\quad u>0,
\end{equation}
where $b\geq0$ is the drift parameter and $\nu$ is a L\'{e}vy measure, i.e.\ a measure on $(0,\infty)$ satisfying $\int_{(0,\infty)}(y\wedge1)\,\nu(\dup y)<\infty$. Our main reference for Bernstein functions and subordination is the monograph \cite{SSV12}. Assume that $S$ and $X$ are independent processes. The \emph{subordinate process} defined by the random time-change
$X_t^\phi:=X_{S_t}$ is again a Markov process; if $X$ has an invariant probability measure $\pi$, then $\pi$ is also invariant for the subordinate process $X^\phi$. This follows immediately from the form of the subordinate Markov transition function which is given by
$$
    P^t_\phi(x,\dup y)=\int_{[0,\infty)}
    P^s(x,\dup y)\,\mu_t(\dup s),
$$
where $\mu_t:=\Prob(S_t\in\cdot)$ is the transition probability of $S_t$; the integral is understood in the sense of vague convergence of probability measures.

We are interested in the following question: Assume that $P^t$ is sub-geometrically convergent to $\pi$ with respect to the $f$-norm as $t\to\infty$; how fast will $P^t_\phi$ tend to $\pi$? More precisely, we need to find a suitable non-increasing function $r_\phi$ on $(0,\infty)$ such that $\lim_{t\to\infty}r_\phi(t)=0$ and
\begin{equation}\label{rate2}
    \left\|P^t_\phi(x,\cdot)-\pi\right\|_f\leq C(x)r_\phi(t),
    \quad x\in E,\;t>0,
\end{equation}
for some positive constant $C(x)$ depending only on $x\in E$. As we will see, if the convergence rates of the original process $X$ are of the three typical forms in \eqref{typical}, then we are able to derive convergence rates for the subordinate Markov process under some reasonable assumptions on the underlying subordinator.

Note that any non-trivial Bernstein function $\phi$ is strictly increasing. In this paper, the inverse function of $\phi$ will be denoted by $\phi^{-1}$.

We can now state the main result of our paper.
\begin{thm}\label{main}
Let $X$ be a Markov process and $S$ an independent subordinator with Bernstein function $\phi$ of the form \eqref{bern-rep}.

\begin{enumerate}
\item[\upshape a)]
    Assume that \eqref{rate1} holds with rate $r(t)=\eup^{-\theta t^\delta}$ for some constants $\theta>0$ and $\delta\in(0,1]$. If $\nu(\dup y)\geq cy^{-1-\alpha}\,\dup y$ for some constants $c>0$ and $\alpha\in(0,1)$, then \eqref{rate2} holds with rate
    $$
        r_\phi(t)=\exp\left[
            -C\,t^{
            \frac{\delta}{\alpha(1-\delta)+\delta}
            }
        \right],
    $$
    where $C=C(\theta,\delta,c,\alpha)>0$.

\item[\upshape b)]
    Assume that \eqref{rate1} holds with rate $r(t)=(1+t)^{-\beta}$ for some constant $\beta>0$. If
    \begin{equation}\label{bern}
        \liminf_{s\to\infty}\frac{\phi(s)}{\log s}>0
        \quad\text{and}\quad
        \liminf_{s\downarrow0}\frac{\phi(\lambda s)}{\phi(s)}>1
        \quad\text{for some {\upshape(}hence, all{\upshape)} $\lambda>1$},
    \end{equation}
    then \eqref{rate2} holds with rate
    $$
        r_\phi(t)=1\wedge\left[
        \phi^{-1}\left(\tfrac 1t\right)
        \right]^\beta.
    $$

\item[\upshape c)]
    Assume that \eqref{rate1} holds with rate $r(t)=[1+\log(1+t)]^{-\gamma}$ for some constant $\gamma>0$. If $\nu(\dup y)\geq cy^{-1-\alpha}\,\dup y$ for some constants $c>0$ and $\alpha\in(0,1)$, then \eqref{rate2} holds with rate
    $$
        r_\phi(t)=1\wedge\log^{-\gamma}(1+t).
    $$
\end{enumerate}
\end{thm}

\begin{rem}
    Typical examples for Bernstein function
    $\phi$ satisfying \eqref{bern} are
    \begin{itemize}
        \item
            $\phi(s)=\log(1+s)$;

        \item
            $\phi(s)=s^\alpha\log^\beta(1+s)$ with $\alpha\in(0,1)$
            and $\beta\in[0,1-\alpha)$;

        \item
            $\phi(s)=s^\alpha\log^{-\beta}(1+s)$ with
            $0<\beta<\alpha<1$;

        \item
            $\phi(s)=s(1+s)^{-\alpha}$ with $\alpha\in(0,1)$.
    \end{itemize}
    We refer to \cite{SSV12} for an extensive
    list of such Bernstein functions.
\end{rem}

Our paper is organized as follows. In order to
prove Theorem~\ref{main}, we establish in Section~\ref{sec2} three types of moment estimates for subordinators; this part is interesting in its own right. The proof of Theorem~\ref{main} will be addressed in Section~\ref{sec3}.  Section~\ref{sec4} contains several concrete models for which the corresponding convergence rates can be explicitly given. For the reader's convenience, the appendix contains some elementary calculations, which have been used in the proof of Theorem~\ref{moments} in Section~\ref{sec2}.

\section{Moment estimates for subordinators}\label{sec2}

In this section we prove some moment estimates for subordinators, which will be crucial for the proof of our main result Theorem~\ref{main}. Related moment estimates for general L\'evy processes and subordinators can be found in \cite[Section 3]{DS15b}. Recently, F.~K\"uhn~\cite{kuehn} extended our results on L\'evy processes to Feller processes.

\begin{thm}\label{moments}
Let $S$ be a subordinator with Bernstein function $\phi$ given by \eqref{bern-rep}.

\begin{enumerate}
\item[\upshape a)]
    Let $\theta>0$ and $\delta\in(0,1]$. If $\nu(\dup y)\geq c\,y^{-1-\alpha}\,\dup y$ for some constants $c>0$ and $\alpha\in(0,1)$, then there exists a $C=C(\theta,\delta,c,\alpha)>0$ such that
    $$
        \EE\,\eup^{-\theta S_t^\delta}
        \leq\exp\left[
            -C\,t^{
            \frac{\delta}{\alpha(1-\delta)+\delta}
            }
        \right]
        \quad\text{for all sufficiently large $t>1$}.
    $$

\item[\upshape b)]
        Let $\beta>0$.
        \begin{enumerate}
        \item[\upshape i)]
            We have
            $$
                \EE S_t^{-\beta}
                \geq
                \frac{1}{\eup\beta\Gamma(\beta)}
                \left[\phi^{-1}\left(\tfrac1t\right)
                \right]^\beta\quad
                \text{for all $t>0$}.
            $$

        \item[\upshape ii)]
            If the Bernstein function $\phi$ satisfies \eqref{bern}, then there exists a $C=C(\beta)>0$ such that
            $$
                \EE S_t^{-\beta}\leq C\left[
                \phi^{-1}\left(\tfrac1t\right)\right]^\beta
                \quad \text{for all sufficiently large $t>1$}.
            $$

        \item[\upshape iii)]
            If the Bernstein function $\phi$ satisfies
            \begin{equation*}
                \liminf_{s\to\infty}
                \frac{\phi(\lambda s)}{\phi(s)}>1
                \quad \text{for some
                {\upshape(}hence, all{\upshape)} $\lambda>1$},
            \end{equation*}
            then there exists a $C=C(\beta)>0$ such that
            $$
                \EE S_t^{-\beta}\leq
                C\left[
                \phi^{-1}\left(\tfrac1t\right)\right]^\beta
                \quad \text{for all $t\in(0,1]$}.
            $$
        \end{enumerate}

\item[\upshape c)]
        Let $\gamma>0$.
        \begin{enumerate}
        \item[\upshape i)]
            If $\nu(\dup y)\geq cy^{-1-\alpha}\,\dup y$ for some constants $c>0$ and $\alpha\in(0,1)$, then there exists a $C=C(\gamma,c,\alpha)>0$ such that
            $$
                \EE\log^{-\gamma}(1+S_t)
                \leq C\log^{-\gamma}\left(1+t^{1/\alpha}\right)
                \quad \text{for all $t>0$}.
            $$

        \item[\upshape ii)]
            If $\nu(\dup y)= cy^{-1-\alpha}\,\dup y$ for some constants $c>0$ and $\alpha\in(0,1)$, then there exists a $C=C(\gamma,c,\alpha)>0$ such that
            $$
                \EE\log^{-\gamma}(1+S_t)
                \geq C\log^{-\gamma}\left(1+t^{1/\alpha}\right)
                \quad \text{for all $t>0$}.
            $$
        \end{enumerate}
        \end{enumerate}
\end{thm}

\begin{rem}\label{momrem}
    Theorem~\ref{moments}\,b) is motivated by an argument in \cite[proof of Theorem 2.1]{BSW11}, where the special case $\beta=1/2$ was treated, cf.\ also \cite[proof of Theorem 1.3]{SSW12}. For the estimate of $\EE S_t^{-1/2}$ for large $t$, it was assumed in \cite[Theorem 2.1]{BSW11} that the Bernstein function $\phi$ satisfies
    \begin{equation}\label{sw1}
        \liminf_{s\to\infty}\frac{\phi(s)}{\log s}>0,
            \quad
        \liminf_{s\downarrow0}\phi(s)|\log s|<\infty,
            \quad
        \limsup_{s\downarrow0}\frac{\phi^{-1}(2s)}{\phi^{-1}(s)}<\infty.
    \end{equation}
    By Lemma~\ref{inverse}\,b) and Lemma~\ref{growth}\,b) below, the third condition in \eqref{sw1} implies that there exist constants $c_1,c_2,\kappa>0$ such that
    $$
        \phi(s)\leq c_1s^\kappa,\quad 0<s\leq c_2,
    $$
    and then
    $$
        \limsup_{s\downarrow0}\phi(s)|\log s|
        \leq
        c_1\limsup_{t\to\infty} \frac{\log t}{t^\kappa}=0.
    $$
    This means that the third condition in \eqref{sw1} implies the second, 
    and so \eqref{sw1} can be written as
    $$
        \liminf_{s\to\infty}\frac{\phi(s)}{\log s}>0,
            \quad
        \limsup_{s\downarrow0}\frac{\phi^{-1}(\lambda s)}{\phi^{-1}(s)}<\infty
            \quad\text{for some (hence, all) $\lambda>1$}.
    $$
\end{rem}

Before we give the proof of Theorem~\ref{moments}, we need some preparations. The following useful lemma is taken from Mu-Fa\ Chen's book \cite[Lemma A.1, p.~193]{Chen05}, see also  \cite[Lemma 5]{SW12} for a special case.

\begin{lem}\label{asdf}
    Let $C>0$, $h:[0,\infty)\to(0,1]$ be an absolutely continuous function, and $\rho:(0,1]\to(0,1]$ be a non-decreasing function. If
    $$
        h'(t)\leq C\rho\big(h(t)\big)
        \quad \text{for almost all $t\geq0$},
    $$
    then
    $$
        G\big(h(t)\big)\leq G\big(h(0)\big)-Ct
        \quad \text{for all $t\geq0$},
    $$
    where
    $$
        G(v):=-\int_v^1\frac{\dup u}{\rho(u)},\quad 0<v\leq 1.
    $$
\end{lem}

For any strictly increasing function $g:(0,\infty)\to(0,\infty)$, we denote by $g^{-1}$ its inverse function.
\begin{lem}\label{inverse}
Let $g:(0,\infty)\to(0,\infty)$ be a strictly increasing function.
\begin{enumerate}
\item[\upshape a)]
    The following statements are equivalent:
    \begin{enumerate}
    \item[\upshape i)]
        $\displaystyle\lim_{t\to\infty}g(t)=\infty$
        \ and \
        $\displaystyle\limsup_{t\to\infty}\frac{g^{-1}(\lambda_0t)}{g^{-1}(t)}<\infty$
        \ for some $\lambda_0>1$.

        \item[\upshape ii)]
        $\displaystyle\lim_{t\to\infty}g(t)=\infty$
        \ and \
        $\displaystyle\limsup_{t\to\infty}\frac{g^{-1}(\lambda t)}{g^{-1}(t)}<\infty$
        \ for all $\lambda>1$.

        \item[\upshape iii)]
            $\displaystyle\liminf_{t\to\infty}\frac{g(\lambda_0t)}{g(t)}>1$
            \ for some $\lambda_0>1$.
        \end{enumerate}
    If $g$ is concave, then \textup{i)--iii)} are also equivalent to:
        \begin{enumerate}
        \item[\upshape iv)]
        $\displaystyle\liminf_{t\to\infty}\frac{g(\lambda t)}{g(t)}>1$
        \ for all $\lambda>1$.
        \end{enumerate}

\item[\upshape b)]
    The following statements are equivalent:
        \begin{enumerate}
        \item[\upshape i)]
        $\displaystyle\limsup_{t\downarrow0}\frac{g^{-1}(\lambda_0t)}{g^{-1}(t)}<\infty$
        \ for some $\lambda_0>1$.

        \item[\upshape ii)]
        $\displaystyle\limsup_{t\downarrow0}\frac{g^{-1}(\lambda t)}{g^{-1}(t)}<\infty$
        \ for all $\lambda>1$.

        \item[\upshape iii)]
        $\displaystyle\liminf_{t\downarrow0}\frac{g(\lambda_0t)}{g(t)}>1$
        \ for some $\lambda_0>1$.
        \end{enumerate}
    If $g$ is concave, then \textup{i)--iii)} are also equivalent to:
        \begin{enumerate}
        \item[\upshape iv)]
        $\displaystyle\liminf_{t\downarrow0}\frac{g(\lambda t)}{g(t)}>1$
        \ for all $\lambda>1$.
        \end{enumerate}
    \end{enumerate}
\end{lem}

\begin{proof}
    We will only show a), since the proof of b) is completely analogous.

\medskip\noindent
    i) $\Leftrightarrow$ ii): The direction ii) $\Rightarrow$ i) is trivial. Conversely, suppose that i) holds true for some $\lambda_0>1$. By the monotonicity of $g$, ii) holds for all $\lambda\in(1,\lambda_0]$. Now assume that $\lambda>\lambda_0$ and let $k:=\lfloor\log_{\lambda_0}\lambda\rfloor$, where $\lfloor x\rfloor$ denotes the integer part of a non-negative real number $x\geq 0$. Since i) implies that there exist $c_1>1$ and $c_2>0$ such that
    \begin{equation}\label{zero1}
            g^{-1}(\lambda_0t)\leq c_1g^{-1}(t),\quad t\geq c_2,
    \end{equation}
    we find
    $$
            g^{-1}(\lambda t)\leq g^{-1}\left(\lambda_0^{k+1}t\right)
            \leq c_1^{k+1}g^{-1}(t),\quad t\geq c_2\lambda_0^{-k},
    $$
    and so
    $$
            \limsup_{t\to\infty}\frac{g^{-1}(\lambda t)}{g^{-1}(t)}
            \leq c_1^{k+1}<\infty.
    $$

\medskip\noindent
    i) $\Leftrightarrow$ iii):
    If i) holds, we can apply $g$ to both sides of \eqref{zero1} and substitute $g^{-1}(t)=s$ to get
    $$
        \lambda_0g(s)\leq g(c_1s),\quad s\geq g^{-1}(c_2).
    $$
    This implies
    $$
        \liminf_{s\to\infty}\frac{g(c_1s)}{g(s)}
        \geq \lambda_0>1,
    $$
    and then iii).

    Conversely, if iii) is satisfied for some $\lambda_0>1$, then it is clear that $\lim_{t\to\infty}g(t)=\infty$, and moreover, we can easily reverse the above argument to deduce i).

\medskip\noindent
    iii) $\Leftrightarrow$ iv) Assume that that $g$ is concave. If  iii) holds for some $\lambda_0>1$, then iv) is true for all $\lambda\geq\lambda_0$. It remains to consider the case $\lambda\in(1,\lambda_0)$. By iii), there exist $c_3>1$ and $c_4>0$ such that
    $$
        g(\lambda_0t)>c_3g(t),\quad t\geq c_4.
    $$
    Using the concavity of $g$, we obtain that for any $t\geq c_4$ and $\lambda\in (1,\lambda_0)$
    \begin{align*}
        g(\lambda t)
        &= g\left(\frac{\lambda_0-\lambda}{\lambda_0-1}\,t + \frac{\lambda-1}{\lambda_0-1}\lambda_0t \right)\\
        &\geq \frac{\lambda_0-\lambda}{\lambda_0-1}\,g(t)+\frac{\lambda-1}{\lambda_0-1}\,g(\lambda_0t)\\
        &> \frac{\lambda_0-\lambda}{\lambda_0-1}\,g(t)+\frac{\lambda-1}{\lambda_0-1}\,c_3g(t),
    \end{align*}
     which yields
     $$
         \liminf_{t\to\infty}\frac{g(\lambda t)}{g(t)}
         \geq\frac{\lambda_0-\lambda}{\lambda_0-1}+\frac{\lambda-1}{\lambda_0-1}\,c_3
         >\frac{\lambda_0-\lambda}{\lambda_0-1}+\frac{\lambda-1}{\lambda_0-1}=1.
     $$
     This completes the proof.
\end{proof}

Below we extend a lemma which was originally proved in \cite{JKLS12}.
\begin{lem}\label{growth}
Let $g:(0,\infty)\to(0,\infty)$ be a non-decreasing function.
\begin{enumerate}
\item[\upshape a)]
        If
        $$
            \liminf_{t\to\infty}\frac{g(\lambda t)}{g(t)} > 1
            \quad\text{for some $\lambda>1$},
        $$
        then there exist positive constants $c_1,\kappa_1,M$ such that
        $$
            g(t)\geq c_1t^{\kappa_1}
            \quad\text{for all $t\in[M,\infty)$}.
        $$

\item[\upshape b)]
        If
        $$
            \liminf_{t\downarrow0}\frac{g(\lambda t)}{g(t)} > 1
            \quad\text{for some $\lambda>1$},
        $$
        then there exist positive constants $c_2,\kappa_2,m$ such that
        $$
            g(t)\leq c_2t^{\kappa_2}\quad \text{for all $t\in(0,m]$}.
        $$
\end{enumerate}
\end{lem}

\begin{proof}
    The first assertion can be found in \cite[Lemma 3.8]{JKLS12}. Part b) can be shown in a similar way. By our assumption, there exist $c_3>1$ and $m>0$ such that
    $$
        g(\lambda t)\geq c_3g(t),\quad 0<t\leq m.
    $$
    This implies that for any $n\in\NN$
    $$
        g(t)\leq c_3^{-n}g(\lambda^nt),
        \quad 0<t\leq\frac{m}{\lambda^{n-1}}.
    $$
    Let $t\in(0,m]$ and set
    $$
        n_t:=\left\lfloor\log_{\lambda}\frac{m}{t}\right\rfloor+1.
    $$
    From this we get
    \begin{gather*}
        g(t)\leq g\left(\frac{m}{\lambda^{n_t-1}}\right)
        \leq c_3^{-n_t} g\left(\lambda^{n_t}\frac{m}{\lambda^{n_t-1}}\right)
        \leq c_3^{-\log_{\lambda}\frac{m}{t}} g(\lambda m)
        = \frac{g(\lambda m)}{m^{\log_{\lambda}c_3}}\,
        t^{\log_{\lambda}m}
    \end{gather*}
    finishing the proof.
\end{proof}

We are now ready to prove Theorem~\ref{moments}.
\begin{proof}[Proof of Theorem~\ref{moments}]
a) We split the proof of this part into four steps.

\medskip\noindent a1)
Without loss of generality, we may assume that $S$ has no drift part, i.e.\ the infinitesimal generator of $S$ is given by
$$
    \mathscr{L}g(x)
    =\int_{(0,\infty)}\big(g(x+y)-g(x)\big)\,\nu(\dup y),\quad g\in C_b^1(\R).
$$
For $\delta\in (0,1]$, set
$$
    g(x):=\eup^{-\theta x^\delta},\quad x\geq0.
$$
By Dynkin's formula, one has
\begin{equation}\label{dynkin}
    \EE\,g(S_t)
    =\EE\,g(S_s)+\EE\left\{\int_s^t\mathscr{L}g(S_u)\,\dup u\right\},\quad 0\leq s\leq t.
\end{equation}

\medskip\noindent a2)
We will now estimate $\mathscr{L}g(x)$ for $x>0$. Since $\nu(\dup y)\geq c\,y^{-1-\alpha}\,\dup y$, we have
\begin{align*}
    \mathscr{L}g(x)
    &=\int_{(0,\infty)} \left(\eup^{-\theta(x+y)^\delta} - \eup^{-\theta x^\delta} \right)\,\nu(\dup y)\\
    &\leq c\,\eup^{-\theta x^\delta} \int_0^\infty\left( \eup^{-\theta x^\delta\left((1+yx^{-1})^\delta-1\right)} -1\right)\,\frac{\dup y}{y^{1+\alpha}}\\
    &=c\,x^{-\alpha}\eup^{-\theta x^\delta} \int_0^\infty\left( \eup^{-\theta x^\delta\left((1+z)^\delta-1\right)} -1\right)\,\frac{\dup z}{z^{1+\alpha}}.
\end{align*}
Noting that
$$
    z\geq\left(\theta^{-1}x^{-\delta}+1\right)^{1/\delta}-1
    \implies
    \eup^{-\theta x^\delta\left((1+z)^\delta-1\right)}-1\leq-\left(1-\eup^{-1}\right),
$$
we conclude that
\begin{equation}\label{generator}
\begin{aligned}
    \mathscr{L}g(x)
    &\leq c\,x^{-\alpha}\eup^{-\theta x^\delta}\int_{\left(\theta^{-1}x^{-\delta}+1 \right)^{1/\delta}-1}^\infty\left( \eup^{-\theta x^\delta\left((1+z)^\delta-1\right)}-1\right)\,\frac{\dup z}{z^{1+\alpha}}\\
    &\leq -\left(1-\eup^{-1}\right)c\,x^{-\alpha}\eup^{-\theta x^\delta} \int_{\left(\theta^{-1}x^{-\delta}+1 \right)^{1/\delta}-1}^\infty\frac{\dup z}{z^{1+\alpha}}\\
    &=-\left(1-\eup^{-1}\right)c\,\alpha^{-1} \eup^{-\theta x^\delta}\left[\left(\theta^{-1}+x^\delta\right)^{1/\delta}-x\right]^{-\alpha}\\
    &=-C_1\rho\big(g(x)\big),
\end{aligned}
\end{equation}
where $C_1:=\left(1-\eup^{-1}\right)c\,\alpha^{-1}\theta^{\alpha/\delta}$ and
$$
    \rho(u)
    :=u\left[(1-\log u)^{1/\delta}-(-\log u)^{1/\delta}\right]^{-\alpha},\quad 0<u\leq1.
$$

\medskip\noindent a3)
Some lengthy, but otherwise elementary, calculations (see Lemma~\ref{convex} in the appendix) yield that $\rho$ is convex and strictly increasing on $(0,1]$. Therefore, \eqref{dynkin} and \eqref{generator} together with Tonelli's theorem and Jensen's inequality give that for $0\leq s\leq t$
\begin{align*}
    \EE\,g(S_t)-\EE\,g(S_s)
    &\leq-C_1 \EE\left\{\int_s^t\rho\big(g(S_u)\big)\,\dup u\right\}\\
    &=-C_1\int_s^t\EE\,\rho\big(g(S_u)\big)\,\dup u\\
    &\leq-C_1\int_s^t\rho\big(\EE\,g(S_u)\big)\,\dup u.
\end{align*}
Setting
$$
    h(t):=\EE\,g(S_t),\quad t\geq0,
$$
we find
$$
    \frac{h(t)-h(s)}{t-s}
    \leq-C_1\frac{1}{t-s} \int_s^t\rho\big(h(u)\big)\,\dup u,
    \quad 0\leq s<t.
$$
Because of \eqref{dynkin} $h$ is absolutely continuous on $[0,\infty)$, and so we can let $t\downarrow s$ to get
$$
    h'(s)
    \leq-C_1\rho\big(h(s)\big)
    \quad \text{for almost all $s\geq 0$}.
$$
According to Lemma~\ref{asdf}, one has
$$
    G\big(h(t)\big)\leq G(1)-C_1t,\quad t\geq0,
$$
where
$$
    G(v)
    :=-\int_v^1\frac{\dup u}{\rho(u)},
    \quad 0<v\leq1.
$$
Clearly, $G$ is strictly increasing with $\lim_{r\downarrow0}G(r)=-\infty$ and $G(1)=0$. Thus, we obtain
\begin{equation}\label{bound11}
    h(t)
    \leq G^{-1}\big(G(1)-C_1t\big)
    = G^{-1}\left(-C_1t\right),
    \quad t\geq0,
\end{equation}
where $G^{-1}$ is the inverse function of $G$.

\medskip\noindent a4)
In order to find a lower bound for $G(v)$, we first observe that for $v\in(0,1]$
\begin{align*}
    G(v)
    &=-\int_v^1u^{-1}\left[(1-\log u)^{1/\delta} - (-\log u)^{1/\delta}\right]^{\alpha}\,\dup u\\
    &=-\int_0^{-\log v}\left[(1+s)^{1/\delta} - s^{1/\delta}\right]^{\alpha}\,\dup s.
\end{align*}
It is easy to see that for $s\geq0$
$$
    (1+s)^{1/\delta}-s^{1/\delta}
    = \frac1\delta\int_s^{1+s} u^{\frac{1-\delta}{\delta}}\,\dup u
    \leq \frac1\delta\,(1+s)^{\frac{1-\delta}{\delta}}
    \leq \frac1\delta\,2^{\frac{1-\delta}{\delta}} \max\left\{s^{\frac{1-\delta}{\delta}},\,1\right\}.
$$
Thus, for $v\in(0,\eup^{-1})$, one has
\begin{align*}
    \int_0^{-\log v}\left[(1+s)^{1/\delta} - s^{1/\delta} \right]^{\alpha}\,\dup s
    &\leq \frac{1}{\delta^\alpha}\,2^{\frac{\alpha(1-\delta)}{\delta}}+\frac{1}{\delta^\alpha}\,2^{\frac{\alpha(1-\delta)}{\delta}}\int_1^{-\log v}s^{\frac{\alpha(1-\delta)}{\delta}}\,\dup s\\
    &\leq C_2+C_2(-\log v)^{\frac{\alpha(1-\delta)+\delta}{\delta}}
\end{align*}
for some $C_2=C_2(\delta,\alpha)>0$. This means that for all $v\in(0,\eup^{-1})$
$$
    G(v)
    \geq -C_2-C_2(-\log v)^{\frac{\alpha(1-\delta)+\delta}{\delta}},
$$
from which one can easily deduce that
$$
    G^{-1}\left(-C_1t\right)
    \leq\exp\left[-\left(C_1C_2^{-1}t-1\right)^{\frac{\delta}{\alpha(1-\delta)+\delta}}\right],\quad t>2C_1^{-1}C_2.
$$
Combining this with \eqref{bound11}, the assertion follows.

\medskip\noindent b)
We prove these three assertions separately.

\medskip\noindent b-i)
Using the identity
\begin{equation}\label{laplace}
    x^{-\beta}
    = \frac{1}{\Gamma(\beta)}\int_0^\infty\eup^{-ux}u^{\beta-1}\,\dup u,
    \quad x\geq0,
\end{equation}
and Tonelli's theorem, we obtain for $t>0$
$$
    \EE S_t^{-\beta}
    = \frac{1}{\Gamma(\beta)} \EE\left\{\int_0^\infty\eup^{-uS_t}u^{\beta-1}\,\dup u\right\}
    = \frac{1}{\Gamma(\beta)} \int_0^\infty\eup^{-t\phi(u)}u^{\beta-1}\,\dup u.
$$
Changing variables according to $v=\phi(u)$, we get
\begin{equation}\label{inte}
    \Gamma(\beta)\EE S_t^{-\beta}
    =\int_0^\infty\eup^{-tv}\left[\phi^{-1}(v)\right]^{\beta-1} \frac{\dup v}{\phi'\left(\phi^{-1}(v)\right)}
    =\frac1\beta\int_0^\infty\eup^{-tv}\,\dup\left\{\left[\phi^{-1}(v)\right]^{\beta}\right\}.
\end{equation}
Thus, we obtain the lower bound, since for any $t>0$
$$
    \beta\Gamma(\beta)\EE S_t^{-\beta}
    \geq\int_0^{1/t} \eup^{-tv}\,\dup\left\{\left[\phi^{-1}(v)\right]^\beta\right\}
    \geq\eup^{-1}\left[\phi^{-1}\left(\tfrac{1}{t}\right)\right]^\beta.
$$

\medskip\noindent b-ii)
By \eqref{bern}, Lemma~\ref{inverse}\,b) and Lemma~\ref{growth}\,b), there exist constants $c_1,c_2>1$ and  $c_3,c_4,\kappa>0$ such that
\begin{align}\label{assu1}
    \phi^{-1}(2s)
    &\leq c_1\phi^{-1}(s), \quad 0<s\leq c_2,
\\ \label{assu2}
    \phi(s)
    &\leq c_3s^\kappa, \quad 0<s\leq c_2,
\\ \label{assu3}
    \phi^{-1}(s)
    &\leq\eup^{c_4s},\quad s\geq c_2.
\end{align}
By \eqref{inte}, the integration by parts formula and \eqref{assu3}, we find for $t>\beta c_4$
$$
    \beta\Gamma(\beta)\EE S_t^{-\beta}
    = \int_0^\infty\left[\phi^{-1}(v)\right]^{\beta} \,\dup\bigl\{\eup^{-tv}\bigr\}
    = \int_0^\infty\left[ \phi^{-1}\left(\tfrac{s}{t}\right) \right]^{\beta}\eup^{-s}\,\dup s.
$$
If $t>c_2^{-1}$ and $s\in(1,c_2t)$, we set $k_s:=\left\lfloor\log_2s\right\rfloor$ and use \eqref{assu1} $1+k_s$ times to get
$$
    \phi^{-1}\left(\tfrac{s}{t}\right)
    \leq\phi^{-1}\left(2^{1+k_s}\tfrac 1{t}\right)
    \leq c_1^{1+k_s}\phi^{-1}\left(\tfrac{1}{t}\right)
    \leq c_1^{1+\log_2s}\phi^{-1}\left(\tfrac{1}{t}\right)
    =c_1\phi^{-1}\left(\tfrac{1}{t}\right)s^{\log_2c_1}.
$$
Thus, for $t>\max\{c_2^{-1},\,2\beta c_4\}$,
\begin{align*}
    \beta\Gamma(\beta)\EE S_t^{-\beta}
    &=\left(\int_0^1+\int_1^{c_2 t} + \int_{c_2 t}^\infty\right)
      \left[ \phi^{-1}\left(\tfrac{s}{t}\right)\right]^{\beta}\eup^{-s}\,\dup s\\
    &\leq\left[\phi^{-1}\left(\tfrac{1}{t}\right)\right]^{\beta}
      + c_1^\beta\left[ \phi^{-1}\left(\tfrac{1}{t}\right) \right]^{\beta}\int_1^{c_2 t} s^{\beta\log_2c_1}\eup^{-s}\,\dup s
      + \int_{c_2 t}^\infty \eup^{-\left(1-\beta c_4/t\right)s}\,\dup s\\
    &\leq \left[\phi^{-1}\left(\tfrac{1}{t}\right)\right]^{\beta}
      + c_1^\beta\,\Gamma\left(1+\beta\log_2c_1\right) \left[\phi^{-1}\left(\tfrac{1}{t}\right)\right]^{\beta}
      + 2\eup^{-c_2 t/2}.
\end{align*}
Since it follows from \eqref{assu2} that
$$
    \limsup_{t\to\infty} \frac{\eup^{-c_2t/(2\beta)}}{\phi^{-1}\left(\frac{1}{t}\right)}
    = \limsup_{s\downarrow0} \frac{1}{s}\exp\left[-\frac{c_2}{2\beta\phi(s)}\right]
    \leq \limsup_{s\downarrow0}\frac{1}{s}\exp\left[-\frac{c_2}{2\beta c_3s^{\kappa}}\right]
    = 0,
$$
we can find some $c_5=c_5(\beta)>0$ such that
$$
    \eup^{-c_2 t/2}
    \leq \left[\phi^{-1}\left(\tfrac{1}{t}\right)\right]^{\beta},
    \quad t>c_5.
$$
Therefore, we conclude that
$$
    \beta\Gamma(\beta)\EE S_t^{-\beta}
    \leq \left(3+c_1^\beta\,\Gamma\left(1+\beta\log_2c_1\right)\right)
         \left[\phi^{-1}\left(\tfrac{1}{t}\right)\right]^{\beta},
         \quad t>\max\left\{c_2^{-1},\,2\beta c_4,\,c_5\right\},
$$
which proves our claim.

\medskip\noindent b-iii)
From our assumption and Lemma~\ref{inverse}\,a) we know
$$
    \lim_{s\to\infty}\phi(s)=\infty,
    \quad\text{and}\quad
    \limsup_{s\to\infty} \frac{\phi^{-1}(2s)}{\phi^{-1}(s)}<\infty.
$$
Thus, there is a constant $c_6>1$ such that
$$
    \phi^{-1}(2s)
    \leq c_6 \phi^{-1}(s)
    \quad\text{for all $s\geq1$},
$$
from which we get
$$
    \phi^{-1}(2^ns)
    \leq c_6^n \phi^{-1}(s)
    \quad \text{for all $s\geq1$ and $n\in\NN$}.
$$
This, together with \eqref{inte}, yields that for any $t\in(0,1]$
\begin{align*}
    \beta\Gamma(\beta)\EE S_t^{-\beta}
    &= \left( \int_0^{1/t} + \sum_{k=0}^\infty \int_{2^k/t}^{2^{k+1}/t} \right)
        \eup^{-tv}\,\dup\left\{\left[\phi^{-1}(v)\right]^\beta\right\}\\
    &\leq\left[\phi^{-1}\left(\tfrac{1}{t}\right)\right]^\beta + \sum_{k=0}^\infty\eup^{-2^k} \left[\phi^{-1}\left(2^{k+1}\tfrac{1}{t}\right)\right]^\beta\\
    &\leq\left(1+\sum_{k=0}^\infty\eup^{-2^k}c_6^{(k+1)\beta
    }\right)\left[\phi^{-1}\left(\tfrac{1}{t}\right)\right]^\beta.
\end{align*}

\medskip\noindent c)
Assume that $\nu(\dup y)\geq cy^{-1-\alpha}\,\dup y$. Using the L\'evy-It\^o decomposition of the L\'evy process $S_t$ we denote by $\widetilde S_t$ that part of $S_t$ whose jumps are governed by the L\'evy measure $cy^{-1-\alpha}\,\dup y$. Clearly, $\widetilde S_t$ has, up to a constant, the same jump behaviour as an $\alpha$-stable subordinator; moreover,
$$
    S_t\geq \widetilde{S}_t,\quad t\geq 0.
$$
By \eqref{laplace} and Tonelli's theorem, we find for $t>0$
\begin{align*}
    \EE\log^{-\gamma}\big(1+S_t\big)
    &\leq \EE\log^{-\gamma}\big(1+\widetilde{S}_t\big)\\
    &= \frac{1}{\Gamma(\gamma)}\EE\left\{ \int_0^\infty\eup^{-u\log\left(1+\widetilde{S}_t\right)}\,u^{\gamma-1}\,\dup u\right\}\\
    &=\frac{1}{\Gamma(\gamma)}\int_0^\infty \EE\left\{\big(1+\widetilde{S}_t\big)^{-u}\right\} u^{\gamma-1}\,\dup u.
\end{align*}
Since $\widetilde S_t$ behaves like an $\alpha$-stable subordinator, we get from \cite[(14)]{BSS03}
$$
    \Prob\big(\widetilde{S}_t\in\dup s\big)
    \leq C_{\alpha,c}\, ts^{-1-\alpha} \eup^{-ts^{-\alpha}}\,\dup s,
$$
where $C_{\alpha,c}>0$ is a constant depending only on $\alpha$ and $c$; using the change of variable $v=ts^{-\alpha}$ and Tonelli's theorem, we obtain
\begin{align*}
    \EE\log^{-\gamma}(1+S_t)
    &\leq \frac{C_{\alpha,c}}{\Gamma(\gamma)} \,t\int_0^\infty\left( \int_0^\infty(1+s)^{-u}s^{-1-\alpha} \eup^{-ts^{-\alpha}}\,\dup s \right)
    u^{\gamma-1}\,\dup u\\
    &= \frac{C_{\alpha,c}}{\alpha\Gamma(\gamma)} \int_0^\infty\left( \int_0^\infty\left(1+\left(\tfrac tv\right)^{1/\alpha} \right)^{-u}\eup^{-v}\,\dup v\right)
            u^{\gamma-1}\,\dup u\\
    &=\frac{C_{\alpha,c}}{\alpha\Gamma(\gamma)} \int_0^\infty\left( \int_0^\infty\left(1+\left(\tfrac tv\right)^{1/\alpha} \right)^{-u}
            u^{\gamma-1}\,\dup u\right) \eup^{-v}\,\dup v\\
    &=\frac{C_{\alpha,c}}{\alpha} \int_0^\infty
    \log^{-\gamma}\left(1+\left(\tfrac tv\right)^{1/\alpha}\right)
    \eup^{-v}\,\dup v\\
    &=\frac{C_{\alpha,c}}{\alpha}\left(\int_0^1
    +\int_1^\infty\right)
    \log^{-\gamma}\left(1+\left(\tfrac tv\right)^{1/\alpha}\right)
    \eup^{-v}\,\dup v\\
    &=:\frac{C_{\alpha,c}}{\alpha}\,(I_1+I_2).
\end{align*}
First,
$$
    I_1\leq\log^{-\gamma}\left(1+t^{1/\alpha}\right)
    \int_0^1
    \eup^{-v}\,\dup v
    =\left(1-\eup^{-1}\right)
    \log^{-\gamma}\left(1+t^{1/\alpha}\right).
$$
Since the function $x\mapsto x^{-1}\log(1+x)$
is strictly decreasing for $x>0$, one has
\begin{equation}\label{efds}
    \frac{\log(1+x)}{x}>\frac{\log(1+\lambda)}{\lambda},
    \quad 0<x<\lambda.
\end{equation}
For $v>1$, using \eqref{efds} with $x=(tv^{-1})^{1/\alpha}$
and $\lambda=t^{1/\alpha}$, we get
$$
    \log\left(1+\left(\tfrac tv\right)^{1/\alpha}\right)
    >\frac{\log\left(1+t^{1/\alpha}\right)}{t^{1/\alpha}}
    \left(\tfrac tv\right)^{1/\alpha}
    =\frac{\log\left(1+t^{1/\alpha}\right)}{v^{1/\alpha}},
$$
and so
$$
    I_2\leq\log^{-\gamma}\left(1+t^{1/\alpha}\right)
    \int_1^\infty v^{\gamma/\alpha}\eup^{-v}\,\dup v
    \leq\Gamma\left(\tfrac \gamma\alpha+1\right)
    \log^{-\gamma}\left(1+t^{1/\alpha}\right).
$$
These estimates give the upper bound in c-i).

If$\,\mbox{}^*$\footnotetext{$^*\,$The lines up to the end of the proof is the version mentioned in the correction to our paper \emph{Adv.~Appl.~Probab.} (2017), 162--181.}
$\nu(\dup y)=cy^{-1-\alpha}\,\dup y$, then for any $t>0$, the distribution of $S_t$ coincides with that of $t^{1/\alpha}S_1$.
Using Lemma \ref{lohgr} with $\tau=t^{1/\alpha}>0$ and $x=S_1\geq 0$, we obtain for $\gamma>0$
\begin{align}
    \mathds{E} \log^{-\gamma}(1+S_t)
    &= \mathds{E} \log^{-\gamma}(1+t^{1/\alpha}S_1)\notag\\
    &\geq \mathds{E}\left[\left((1+S_1)\log(1+t^{1/\alpha})\right)^{-\gamma}\right]\label{bigtime}\\
    &=\mathds{E}\left[(1+S_1)^{-\gamma}\right]\cdot\log^{-\gamma}\left(1+t^{1/\alpha}\right).\notag
\end{align}
From \eqref{bigtime} we get the desired lower bound in c-ii).
\end{proof}

\section{Proof of Theorem~\ref{main}}\label{sec3}

Theorem~\ref{main} follows at once from Lemma~\ref{general} below together with the corresponding upper moment bounds for subordinators derived in Theorem~\ref{moments}.

\begin{lem}\label{general}
    If \eqref{rate1} holds with some rate function $r$, then so does \eqref{rate2} with rate function $r_\phi(t)=\EE\,r(S_t)$.
\end{lem}

\begin{proof}
    By the definition of $P^t_\phi(x,\dup y)$ and \eqref{rate1}, we find that for all $t>0$ and $x\in E$
    \begin{align*}
        \left\|P^t_\phi(x,\cdot)-\pi\right\|_f
        &=\left\| \int_{[0,\infty)}\big(P^s(x,\cdot)-\pi\big)\, \mu_t(\dup s) \right\|_f\\
        &\leq\int_{[0,\infty)}\left\|P^s(x,\cdot)-\pi\right\|_f\, \mu_t(\dup s)\\
        &\leq C(x)\int_{[0,\infty)}r(s)\,\mu_t(\dup s)\\
        &=C(x)\EE\,r(S_t),
    \end{align*}
    and the claim follows.
\end{proof}

\section{Examples}\label{sec4}

In this section, we discuss three models for which we are able to obtain explicit convergence rates, so that our main result Theorem~\ref{main} can be applied.


\subsection{$Q$-processes on $\Z_+$}

Let $Q=(q_{i j})_{i, j\in\Z_+}$ be a $Q$-matrix with
$$
    q_{00}=-\lambda_0,
    \quad q_{0 i}=\lambda_0 p_i,
    \quad q_{i 0}=-q_{i i}=\lambda_i,
    \quad i\geq 1,
$$
and $q_{i j}=0$ otherwise, where $(p_i)_{i\geq1}$ and $(\lambda_i)_{i\geq0}$ are two sequences of positive numbers with $\sum_{i=1}^{\infty}p_i=1$ and $\sup_{i\geq0}\lambda_i<\infty$. It is well known that there exists a unique $Q$-process with transition semigroup $P^t=\sum_{n=0}^{\infty}(t Q)^n/n!$, cf.\ \cite[Corollary 2.24]{che04}. Moreover, we assume
$$
    \liminf_{i\to\infty}\lambda_i=0
    \quad\text{and}\quad
    \sum_{i=1}^{\infty}p_i\lambda_i^{-1}<\infty.
$$
Under these assumptions it is easy to see that the process admits an invariant distribution $\pi$ given by
$$
    \pi_0
    =\bigg(1+\lambda_0\sum_{j=1}^{\infty} p_j\lambda_j^{-1}\bigg)^{-1},
    \quad
    \pi_i
    = \pi_0\lambda_0 p_i\lambda_i^{-1},
    \quad i\geq1.
$$

For this toy model, it is known,
see \cite[(the proofs of) Propositions 12 and 14]{for05}, that
\begin{enumerate}
\item[a)]
    If
    $$
        \sum_{i=1}^{\infty}p_i\left(1\vee\lambda_i^{-1} \right)\lambda_i^{-1/2}\eup^{\theta^2\lambda_i^{-1}} <\infty
    $$
    for some $\theta>0$, then \eqref{rate1} holds for any $q\in[0,1]$ with
    $$
        r(t)=\eup^{-2\theta q \sqrt{t}},\quad f(i)= \left(1+\lambda_i^{-1/2}\eup^{\theta^2\lambda_i^{-1}} \right)^{1-q}.
    $$

\item[b)]
    If
    $$
        \sum_{i=1}^{\infty}p_i\lambda_i^{-1-\theta}<\infty
    $$
    for some $\theta\geq0$, then \eqref{rate1} holds for any $\beta\in[0,\theta]$ with
    $$
        r(t)=(1+t)^{-\beta},\quad f(i)=1+\lambda_i^{\beta-\theta}.
    $$

\item[c)]
    If
    $$
        \sum_{i=1}^{\infty}p_i\left( 1\vee\lambda_i^{-1}\right)\log^{\theta}\left( 1\vee\lambda_i^{-1}\right)<\infty
    $$
    for some $\theta\geq0$, then \eqref{rate1} holds for any $\gamma\in[0,\theta]$ with
    $$
        r(t)=\left[1+\log(1+t)\right]^{-\gamma},\quad f(i)=\left[1+\log\left(1\vee\lambda_i^{-1} \right)\right]^{\theta-\gamma}.
    $$
\end{enumerate}

\subsection{Diffusion processes on $\R^d$}

Consider the stochastic differential equation (SDE):
\begin{equation}\label{e1}
    \dup X_t = b(X_t)\,\dup t+\sigma(X_t)\,\dup B_t,
    \quad X_0=x,
\end{equation}
where $\{B_t:t\geq0\}$ is a standard $d$-dimensional Brownian motion, $b:\R^d\to\R^d$ and $\sigma:\R^d\to\R^{d\times d}$ are locally Lipschitz, $\sigma$ is bounded, and the smallest eigenvalue of $a(x):=\sigma(x)\sigma^\top(x)$ is bounded away from zero in every bounded domain. If there exist constants $p\in(0,1)$, $C>0$ and $M>0$ such that
$$
    \left\langle b(x), x\right\rangle\leq-C|x|^{1-p},\quad |x|\geq M,
$$
then \eqref{e1} has a unique solution with infinite life-time and an invariant probability measure $\pi$; moreover, \eqref{rate1} holds for any $q\in(0,1)$ with
$$
    r(t)
    = \exp\left[-C_1qt^{(1-p)/(1+p)} \right],
    \quad
    f(x)
    = 1+(1+|x|)^{-2p(1-q)}\exp\left[C_2(1-q)|x|^{1-p}\right]
$$
for some positive constants $C_1$ and $C_2$, see \cite[Theorem 5.4]{dou09}.

Set
\begin{gather*}
    \smash[b]{K:=\frac{\lambda_{+}-\lambda_{-}+\Lambda}{2\lambda_{+}},}
\intertext{where}
    \lambda_+:=\sup_{x\neq0}\frac{\langle a(x)x,x\rangle}{|x|^2},
    \quad
    \lambda_-:=\inf_{x\neq0} \frac{\langle a(x)x,x\rangle}{|x|^2},
    \quad
    \Lambda:=\sup_{x\in\R^d}\operatorname{Tr}\big(a(x)\big).
\end{gather*}
In order to obtain algebraic rates of convergence, we need the following condition:
\begin{enumerate}
\item[(A)]
    There exist constants $C>K$ and $M>0$ such that
    $$
        \langle b(x), x\rangle
        \leq -C\lambda_{+},\quad |x|\geq M.
    $$
\end{enumerate}
It is not hard to see that under \textup{(A)} there is a unique non-explosive solution to the SDE \eqref{e1}.

\begin{prop}\label{prop1}
    Assume that \textup{(A)} holds. Then \eqref{rate1} holds for any $\beta\in(0,C-K)$ and any $m\in(0,C-K-\beta)$ with
    $$
        r(t)=(1+t)^{-\beta},
        \quad
        f(x)=1+|x|^{2m}.
    $$
\end{prop}

\subsection{SDEs driven by $\alpha$-stable processes}

Consider the SDE
\begin{equation}\label{levy}
    \dup X_t=b(X_t)\,\dup t+\dup Z_t,\quad X_0=x,
\end{equation}
where $\{Z_t:t\geq0\}$ is an $\alpha$-stable ($0<\alpha<2$) L\'{e}vy process on $\R^d$, and $b:\R^d\to\R^d$ is continuous such that
$$
    \langle b(x)-b(y),x-y\rangle\leq L|x-y|^2,\quad x,y\in\R^d
$$
for some $L\in\R$. Under these assumptions there exists a unique non-explosive solution to the SDE \eqref{levy}, which is (strong) Feller by the dimension-free Harnack inequality, cf.\ \cite{WW14, Den14}, and Lebesgue irreducible, see e.g.\ \cite{mas07}.

\begin{prop}\label{prop2}
\textup{a)}
    If there exist constants $p\geq0$, $C>0$ and $M>0$ such that
    \begin{equation}\label{ss33}
        \langle b(x), x\rangle
        \leq -C|x|^2\log^{-p}(1+|x|),
        \quad |x|\geq M,
    \end{equation}
    then \eqref{rate1} holds for any $q\in(0,1)$ and any $m\in(0,1\wedge(d-1+\alpha))$ with
    $$
        r(t)=\exp\left[-\widetilde{C}qt^{1/(1+p)}\right],
        \quad f(x)=1+|x|^{m(1-q)}
    $$
    for some constant $\widetilde{C}=\widetilde{C}(C,p,M,m)>0$.

\medskip\noindent\textup{b)}
    If there exist $p\in(0,1\wedge(d-1+\alpha))$, $C>0$ and $M>0$ such that
    \begin{equation}\label{ss44}
        \langle b(x), x\rangle\leq -C|x|^{p+2-1\wedge(d-1+\alpha)},
        \quad |x|\geq M,
    \end{equation}
    then \eqref{rate1} holds for any $\beta\in\left(0,\frac{p}{1\wedge(d-1+\alpha)-p}\right)$ and any $0<m<p-\beta(1\wedge(d-1+\alpha)-p)$  with
    $$
        r(t)=(1+t)^{-\beta},
        \quad
        f(x)=1+|x|^m.
    $$
\end{prop}

\begin{rem}
    If \eqref{ss33} holds with $p=0$, then we get exponential rates of convergence in Proposition~\ref{prop2}\,a),
    see e.g.\ \cite[Lemma 2.4]{mas07}.
\end{rem}

\subsection{Proofs of Propositions~\ref{prop1} and~\ref{prop2}}

Our proofs of Propositions~\ref{prop1} and~\ref{prop2} are based on the following Foster--Lyapunov criterion, which is a special case of \cite[Theorem 3.2]{dou09}, see also \cite[Theorem 2.8]{DFMS04} for the corresponding result for discrete-time Markov chains.

Let $X$ be a Markov process with generator $(A,\mathcal D)$. It is well known that for $g:=Af$ and $f\in\mathcal D$ the process $M_t^f := f(X_t) - f(X_0) - \int_0^t g(X_s)\,\dup s$ is a martingale. Recall that the extended generator consists of all pairs $(f,g)$ of measurable functions such that $M_t^f$ is a local martingale for some unique $g$; note that we do not require $f\in\mathcal D$ nor $g=Af$, see \cite[p.\ 25, 26]{BSW14} for details. We denote the extended generator by $(\mathscr{A},D(\mathscr{A}))$.

\begin{prop}\label{drift}
    Let $X$ be a Markov process on the state space $(E,\mathscr{B}(E))$ with extended generator $(\mathscr{A},D(\mathscr{A}))$. Assume that
    \begin{enumerate}
        \item[\upshape a)]
             Some skeleton chain {\upshape(}i.e.\ a Markov chain
             with transition kernel $P^T$ for
             some $T>0${\upshape)} is $\psi$-irreducible
             for some $\sigma$-finite measure $\psi$.
        \item[\upshape b)]
            There exist a closed petite set $B$, a constant $b>0$, a continuous function $V\in D(\mathscr{A})$, $V: E\to[1, \infty)$ with $\sup_B V<\infty$, and a non-decreasing differentiable concave function $\varphi:[1, \infty)\to(0, \infty)$ satisfying $\lim_{x\to\infty}\varphi'(x)=0$ and
            \begin{equation}\label{driftineq}
                \mathscr{A} V(x)
                \leq-\varphi\circ V(x)+b\I_B(x),\quad x\in E.
            \end{equation}
    \end{enumerate}
    Then there exists an invariant probability measure $\pi$ such that $\pi(\varphi\circ V)<\infty$, and \eqref{rate1} holds for any $q\in(0,1)$ with
    $$
        r(t)
        = 1 \wedge \left(\varphi\circ H_{\varphi}^{-1}(t) \right)^{-q},
        \quad\text{and}\quad
        f= 1 \vee (\varphi\circ V)^{1-q},
    $$
    where $H_\varphi^{-1}$ is the inverse of the following function
    $$
        H_\varphi(u)
        = \int_1^u\frac{\dup x}{\varphi(x)},
        \quad u\geq 1.
    $$
\end{prop}

\begin{proof}[Proof of Proposition~\ref{prop1}]
As it is pointed out in \cite[p.\ 908]{dou09}, it is a standard argument that \textup{(A)} ensures the existence of a unique invariant probability measure $\pi$, and that any skeleton chain is $\pi$-irreducible. Thus, we know that every compact set is a closed petite set, cf.\ \cite[Theorems 5.1 and 7.1]{twe94}.

Fix $\beta\in(0,C-K)$ and $m\in(0,C-K-\beta)$. Observe that $C^2(\R^d)\subset D(\mathscr{A})$. Choose a test function $V\in C^2(\R^d)$ such that $V(x)=\left(1+|x|\right)^{2m+2\beta+2}$ for $|x|>M$. By \textup{(A)}, we get for all $|x|>M$
\begin{align*}
    \mathscr{A} V(x)
    &\leq-C\lambda_{+}\left(2m+2\beta+2\right) \left(1+|x|\right)^{2m+2\beta+1}|x|^{-1}\\
    &\qquad\mbox{}+\lambda_{+}\left(m+\beta+1\right)(2m+2\beta+1)\left(1+|x|\right)^{2m+2\beta}\\
    &\qquad\mbox{}-\lambda_{-}\left(m+\beta+1\right)\left(1+|x|\right)^{2m+2\beta+1}|x|^{-1}\\
    &\qquad\mbox{}+\Lambda \left(m+\beta+1\right)\left(1+|x|\right)^{2m+2\beta+1}|x|^{-1}\\
    &\leq-2\lambda_+(m+\beta+1)(C-K-\beta-m)(1+|x|)^{2m+2\beta}\\
    &=:-C_1(1+|x|)^{2m+2\beta}.
\end{align*}
This implies that the Foster-Lyapunov condition \eqref{driftineq} holds with $\varphi(x)= C_1 x^{\frac{m+\beta}{m+\beta+1}}$, $B=\{x\in\R^d: |x|\leq M\}$ and some constant $b>0$. Using Proposition~\ref{drift} with $q=\beta/(m+\beta)$, we know that \eqref{rate1} holds with
$$
    r(t)= 1 \wedge \left[C_1^{-\frac{\beta}{m+\beta}}\left( 1+\frac{C_1}{m+\beta+1}\,t \right)^{-\beta}\right]
    \quad\text{and}\quad
    f(x)= 1\vee \left[C_1^{\frac{m}{m+\beta}}(1+|x|)^{2m}\right].
$$
This finishes the proof.
\end{proof}

\begin{proof}[Proof of Proposition~\ref{prop2}]
    a) First of all, it is clear that
    $$
        \left\{V\in C^2(\R^d)\,:\,\left|\int_{|y|>1}\left(V(x+y)-V(x)\right) \frac{\dup y}{|y|^{d+\alpha}} \right|<\infty
        \quad \text{for all $x\in\R^d$}\right\}
        \subset D(\mathscr{A}).
    $$
    Fix $m\in(0,1\wedge(d-1+\alpha))$ and choose a test function $V\in C^2(\R^d)$ such that $V(x)=(1+|x|)^m$ for $|x|>M$. Then $V\in D(\mathscr{A})$ and
    \begin{gather*}
        \mathscr{A}V(x)=\langle b(x),\nabla V(x)\rangle+ \mathscr{A}_1V(x)+\mathscr{A}_2V(x),
    \intertext{where}
        \mathscr{A}_1V(x):=\int_{|y|>1}\left( V(x+y)-V(x)\right)\frac{c_{d,\alpha}}{|y|^{d+\alpha}} \,\dup y,\\
        \mathscr{A}_2V(x):=\int_{0<|y|\leq1}\left( V(x+y)-V(x)-\langle \nabla V(x),y\rangle\right)\frac{c_{d,\alpha}}{|y|^{d+\alpha}} \,\dup y.
    \end{gather*}
    Since $|\mathscr{A}_1V(x)|=o(1)$ and $|\mathscr{A}_2V(x)|=o(1)$ as $|x|\to\infty$, see the proof of \cite[Lemma 2.4]{mas07}, we obtain from \eqref{ss33} that for $|x|>M$
    \begin{align*}
        \mathscr{A}V(x)
        &\leq-C m\left(1+|x|\right)^{m-1}|x| \log^{-p}\left(1+|x|\right)+o(1)\\
        &\leq-C_1V(x)\left(1+p+\log V(x)\right)^{-p}
    \end{align*}
    for some $C_1=C_1(C,p,M,m)>0$. Hence, \eqref{driftineq} is satisfied with $\varphi(x)=C_1x(1+p+\log x)^{-p}$, and the claim follows from Proposition~\ref{drift} and some straightforward calculations.

\medskip\noindent b)
    Set $\varrho:=1\wedge(d-1+\alpha)-p$ and fix $\beta\in(0,p/\varrho)$ and $m\in(0,p-\beta\varrho)$. Choose $V\in C^2(\R^d)$ such that $V(x)=(1+|x|)^{m+\varrho(\beta+1)}$ for $|x|>M$. As in part a), it is not difficult to obtain from \eqref{ss44} that \eqref{driftineq} holds with $\varphi(x)=C_2x^{\frac{m+\varrho\beta} {m+\varrho(\beta+1)}}$ for some $C_2=C_2(C,p,M,m)>0$. Thus, the assertion follows from Proposition~\ref{drift} with $q=\varrho\beta/(m+\varrho\beta)$.
\end{proof}

\appendix

\section{}\label{app}

\begin{lem}\label{convex}
    Let $\tau\geq1$, $\alpha\in(0,1)$ and
    $$
        g(x)
        =x\left[(1-\log x)^\tau-(-\log x)^\tau\right]^{-\alpha},
        \quad 0<x\leq1.
    $$
    Then the function $g$ is convex and strictly increasing on $(0,1]$.
\end{lem}

\begin{proof}
    Obviously, we only need to prove the statement for $x\in(0,1)$. By a direct calculation, we find that for $x\in(0,1)$
    \begin{align*}
        g'(x)
        &= \frac{g(x)}{x} + \tau\alpha\frac{g(x)}{x} \left[(1-\log x)^\tau - (-\log x)^\tau \right]^{-1}
           \left[(1-\log x)^{\tau-1}- (-\log x)^{\tau-1} \right]\\
        &= \xi(x)\eta(x),
   \end{align*}
    where
    \begin{gather*}
        \xi(x)
        :=\frac{g(x)}{x}
        = \left[(1-\log x)^\tau-(-\log x)^\tau\right]^{-\alpha},\\
        \eta(x)
        :=1 + \tau\alpha\left[(1-\log x)^\tau - (-\log x)^\tau \right]^{-1}
        \left[(1-\log x)^{\tau-1} - (-\log x)^{\tau-1}\right].
    \end{gather*}
    Since $\xi(x)>0$ and $\eta(x)>0$ for all $x\in(0,1)$, $g$ is strictly increasing on $(0,1)$. Noting that $g''(x)=\xi'(x)\eta(x)+\xi(x)\eta'(x)$, it suffices to prove that $\xi'(x)\geq 0$ and $\eta'(x)\geq 0$ for all $x\in(0,1)$. For $x\in(0,1)$, one has
    $$
         \xi'(x)
         =\frac{\tau\alpha}{x} \left[(1-\log x)^\tau- (-\log x)^\tau \right]^{-\alpha-1} \left[(1-\log x)^{\tau-1}- (-\log x)^{\tau-1} \right]
         \geq 0,
    $$
    and
    \begin{align*}
        \frac{1}{\tau\alpha}\eta'(x)
        &= \frac{\tau}{x} \left[(1-\log x)^\tau- (-\log x)^\tau \right]^{-2} \left[(1-\log x)^{\tau-1}- (-\log x)^{\tau-1} \right]^2\\ &\qquad\mbox{}-\frac{\tau-1}{x} \left[(1-\log x)^\tau- (-\log x)^\tau \right]^{-1} \left[(1-\log x)^{\tau-2}- (-\log x)^{\tau-2} \right]\\
        &= \frac{1}{x} \left[(1-\log x)^\tau- (-\log x)^\tau \right]^{-2}\zeta(x),
    \end{align*}
    where
    \begin{align*}
        \zeta(x)
        &:= \tau\left[(1-\log x)^{\tau-1}- (-\log x)^{\tau-1} \right]^2\\
        &\qquad\mbox{} -(\tau-1) \left[(1-\log x)^\tau- (-\log x)^\tau \right] \left[(1-\log x)^{\tau-2}- (-\log x)^{\tau-2} \right]\\ &= (1-\log x)^{2\tau-2} +(-\log x)^{2\tau-2} -2\tau(1-\log x)^{\tau-1}(-\log x)^{\tau-1}\\
        &\qquad\mbox{}+ (\tau-1)(1-\log x)^{\tau}(-\log x)^{\tau-2} +(\tau-1)(1-\log x)^{\tau-2}(-\log x)^{\tau}.
    \end{align*}
    It remains to check that $\zeta(x)\geq0$ for all $x\in(0,1)$. From the elementary inequality
    $$
        y^2+z^2
        \geq
        2yz,\quad y,z\geq 0,
    $$
    it follows that for $x\in(0,1)$
    \begin{align*}
         \zeta(x)
         &\geq 2(1-\log x)^{\tau-1}(-\log x)^{\tau-1} - 2\tau(1-\log x)^{\tau-1}(-\log x)^{\tau-1}\\
         &\qquad\mbox{} + (\tau-1)(1-\log x)^{\tau}(-\log x)^{\tau-2} + (\tau-1)(1-\log x)^{\tau-2}(-\log x)^{\tau}\\
         &=(\tau-1)(1-\log x)^{\tau-2}(-\log x)^{\tau-2}\\
         &\geq0,
    \end{align*}
    and the proof is finished.
\end{proof}

\begin{lem}\label{lohgr}
Let$\,\mbox{}^*$\footnotetext{$^*\,$This is the version of Lemma A.2 as mentioned in the correction to our paper \emph{Adv.~Appl.~Probab.} (2017), 162--181.} $\tau> 0$ and $x\geq 0$. Then
\begin{gather*}
    \log(1+\tau x) \leq (1+x)\log(1+\tau).
\end{gather*}\end{lem}

\begin{proof}
    By Bernoulli's inequality,
    $$
        1+\tau x
        \leq 1+\tau(1+x)
        \leq (1+\tau)^{1+x}
    $$
    and the claim follows upon taking logarithms on both sides.
\end{proof}


\acks
C.-S. Deng gratefully acknowledges support through the Alexander-von-Humboldt foundation, the National Natural Science Foundation of China (11401442) and the China Postdoctoral
Science Foundation (2013). Y.-H. Song gratefully acknowledges support through
the National Natural Science Foundation of China (11426219, 11501576)
and the China Scholarship Council (201407085008).


%
%
%
%

\end{document}